\documentclass[sn-mathphys]{sn-jnl}% Math and Physical Sciences Reference Style
\jyear{2021}%

\theoremstyle{thmstyleone}%
\newtheorem{theorem}{Theorem}%  meant for continuous numbers
\newtheorem{proposition}[theorem]{Proposition}%
\usepackage{epstopdf}
\usepackage{mathrsfs}%%%
\usepackage{amsmath}
\usepackage{graphicx}
\usepackage{eepic}
\usepackage{subfigure}
\usepackage{amsfonts}
\usepackage[figuresright]{rotating}
\usepackage{verbatim}
\usepackage{ragged2e}
\usepackage{subfigure}
\usepackage{multirow}
\usepackage{multicol}
\usepackage[numbers,sort&compress]{natbib}% 参考文献连字符，自动排序
\usepackage{arydshln}
\usepackage{float}
\usepackage{dsfont}

\theoremstyle{thmstyletwo}%
\newtheorem{example}{Example}%
\newtheorem{remark}{Remark}%

\theoremstyle{thmstylethree}%
\newtheorem{definition}{Definition}%
\newtheorem{corollary}{Corollary}
\newtheorem{lemma}{Lemma}
\newtheorem{assumption}{Assumption}
\begin{document}

\title[Adaptive smoothing mini-batch stochastic accelerated gradient method]{Adaptive smoothing mini-batch stochastic accelerated gradient method for nonsmooth convex stochastic composite optimization}


\author[1]{\fnm{Ruyu} \sur{Wang}}\email{wangruyu@bjtu.edu.cn}

\author*[1]{\fnm{Chao} \sur{Zhang}}\email{zc.njtu@163.com}

\affil[1]{\orgdiv{Department of Applied Mathematics}, \orgname{Beijing Jiaotong University}, \orgaddress{\street{No.3 Shangyuancun}, \city{Haidian District}, \postcode{100044}, \state{Beijing}, \country{China}}}


\abstract{This paper considers a class of convex constrained nonsmooth convex stochastic composite optimization problems whose objective function is given by the summation of a differentiable convex component, together with a general nonsmooth but convex component. The nonsmooth component is not required to have easily obtainable proximal operator, or have the max structure that the smoothing technique in \cite{Nesterov} can be used. In order to solve such type problems, we propose an adaptive smoothing mini-batch stochastic accelerated gradient (AdaSMSAG) method, which combines the stochastic approximation method, the Nesterov's accelerated gradient method \cite{Nesterov}, and the smoothing methods that allow general smoothing approximations. Convergence of the method is established. Moreover, the order of the worst-case iteration complexity is better than that of the state-of-the art stochastic approximation methods. Numerical results are provided to illustrate the efficiency of the proposed AdaSMSAG method for a risk management in portfolio optimization and a family of Wasserstein distributionally robust support vector machine  problems with real data.}




\maketitle

\section{Introduction}

In this paper, we consider the constrained nonsmooth convex stochastic composite minimization problem
\begin{eqnarray}\label{orip}
\psi^*:=\min\limits_{x\in X}\{\psi(x): =  f(x) + h(x)\},
\end{eqnarray}
where \begin{eqnarray}\label{orip-2}
f(x)=E\left[F(x, \xi)\right],~~h(x)=E\left[H(x, \xi)\right].
\end{eqnarray}
Here $X$ is a closed convex set in the Euclidean space $\mathds{R}^n$, $\xi$ is a random vector whose probability distribution is supported on $\Xi\subseteq\mathds{R}^d$, and $F(\cdot,\xi): X\rightarrow\mathds{R}$ and $H(\cdot,\xi): X\rightarrow\mathds{R}$ are functions such that for every $\xi\in\Xi$, $F(\cdot, \xi) $ is convex with $L$-Lipschitz continuous gradient $\nabla F(\cdot,\xi)$, and $H(\cdot, \xi)$ is nonsmooth convex. We do not assume the proximal operator of $H(\cdot, \xi):X\rightarrow\mathds{R}$ and consequently the proximal operator of $h$ are easy to compute.
Many applications especially in machine learning are in this setting. There are two main difficulties to design an efficient algorithm for \eqref{orip}: the general nonsmoothness of $h$, and the stochastic setting that leads the difficulty for computing the objective value and the gradient of $\psi$ for the random vector $\xi$ in the general support set $\Xi$.

In the deterministic setting, i.e., $\Xi$ is a singleton in \eqref{orip-2},  smoothing algorithms that utilize the structure of the problems to define smoothing functions  have been shown to be an effective way to overcome the nonsmoothness in optimization.

Nesterov introduced a smoothing scheme \cite{Nesterov} for nonsmooth convex problems with additional assumptions that $X$ is bounded, and $h(x)$ is a nonsmooth convex function with explicit max-structure as follows.
\begin{eqnarray}\label{h}
h(x) = \max\limits_{u\in U} \{\langle Ax, u\rangle - Q(u) \},
\end{eqnarray}
where $U$ is a bounded closed convex set, $A$ is a linear operator, and $Q(\cdot)$ is a continuous convex function. In \cite{Nesterov}, the max structure in \eqref{h} is necessary to construct the smooth approximation $\tilde{h}_{\mu}(x)$ of $h(x)$ with its gradient $\nabla \tilde{h}_{\mu}(x)$ being Lipschitz continuous. This method combines the smooth approximation with a fixed predetermined smoothing parameter $\mu$, and the Nesterov's accelerated gradient scheme \cite{Nesterov1,Nesterov2} for smooth convex optimization problems, to achieve the worst-case iteration complexity
$$O\left(\frac{1}{\epsilon}\right),$$
where $\epsilon$ is the desired absolute accuracy of the difference between the original function value at the approximate solution and the the dual function value at the corresponding approximate dual solution. It was shown that the order $O\left(\frac{1}{\epsilon^2}\right)$ of the worst-case iteration complexity cannot be improved for the simplest subgradient method in nonsmooth optimization \cite{NemirovskiY}, and Nesterov pointed out that the improvement order of $O\left(\frac{1}{\epsilon}\right)$ is beneficial from the proper use of the structure of the problem \cite{Nesterov}. Quoc proposed an adaptive smoothing proximal gradient method \cite{Quoc} that updates $\mu$ from $\mu_{k}$ to $\mu_{k+1}$ at each iteration for unconstrained nonsmooth composite convex minimization problems with
$$X=\mathds{R}^n,\ f(x) \equiv 0,\ h(x) = h_1(x) + h_2(x),$$
where $h_1(x)$ is a nonsmooth convex function with explicit max-structure and $h_2(x)$ is a simple nonsmooth convex function such that the proximal operator is easy to compute. This method automatically updates the smoothing parameter at each iteration by
$$\mu_{k+1}=\frac{\overline{c}\mu_1}{k+\overline{c}},$$
where $k\geq1$, $\overline{c}\ge 1$ is a constant and $\mu_1$ is an initial smoothing parameter. It achieves the $\mathcal{O}\left(\frac{1}{\epsilon}\right)$-worst-case iteration complexity as in \cite{Nesterov}, and the adaptive updating of the smoothing parameter has been demonstrated efficiency in numerical experiments.

%实际上光滑化方法很多，引用一般光滑化方法的文章
Nesteorv's smoothing method \cite{Nesterov} and Quoc's adaptive smoothing proximal gradient method \cite{Quoc} require the nonsmooth term of the objective function to have the explicit max-structure, or easily obtainable proximal operator. Fortunately, a type of smoothing methods for constrained nonsmooth optimization problems employs the smoothing approximations without such requirements, such as the smoothing projected gradient method \cite{Zhang2}, the smoothing quadratic regularization method \cite{BianChen}, the smoothing sequential quadratic programming method \cite{XuYeZhang}, the smoothing sequential quadratic programming framework \cite{Liu}, and the smoothing active set method in \cite{Zhang1}.
At each step of a certain smoothing method, the objective function is approximated by a smooth function with a fixed smoothing parameter and a certain smooth algorithm is employed to get the next iterate point. A special updating rule for the smoothing parameter is then checked to determine whether to decrease the smoothing parameter or keep it unchanged. More information about this type of smoothing methods and the various smoothing approximations can be found in the excellent survey paper \cite{Chen1}.

For the stochastic setting we focus on in this paper, the stochastic approximation (SA) method is one important approach for solving the stochastic convex programming, which can be dated back to the pioneering paper by Robbins and Monro \cite{Robbins}. A robust version of the SA method developed by Polyak \cite{Polyak}, and Polyak and Juditsky \cite{Polyak2}, improves the original version of the SA method. It has clear that for nonsmooth stochastic convex optimization, the order of iteration complexity required to find an $\epsilon$-approximate solution $\bar{x}$, i.e. $\psi(\bar{x})-\psi^* \le \epsilon$ for a pre-determined accuracy $\epsilon>0$, cannot be smaller than
$$O\left(\frac{1}{\epsilon^2}\right),$$
as pointed out in \cite{Nemirovski}. By combining the Nesterov's accelerated gradient scheme \cite{Nesterov1,Nesterov2} and the SA method, Lan \cite{Lan1} proposed an accelerated SA (AC-SA) method to the nonsmooth stochastic convex problems in a convex compact set, which achieves the order $\mathcal{O}\left(\frac{1}{\epsilon^2}\right)$ of iteration complexity. Ghadimi et al. \cite{Ghadimi1} extended this method to the nonsmooth stochastic strongly convex problems. In 2016, Ghadimi et al. \cite{Ghadimi2} proposed an accelerated gradient (AG) method for the unconstrained nonconvex nonsmooth optimization problems, where the nonsmooth component is required to be a simple convex function such as $\|x\|_1$, for which the proximal operator is easy to compute. In 2016, Wang et al. \cite{wangxiao} proposed a class of penalty methods with stochastic approximation for solving stochastic nonlinear programming problems. In 2021, Bai et al. \cite{Bai} developed a symmetric accelerated stochastic alternating direction method of multipliers (SAS-ADMM) for solving separable convex optimization problems with linear constraints. In 2021, Wang et al. \cite{Wang} proposed a mini-batch stochastic Nesterov's smoothing method for the constrained convex stochastic composite optimization problems. The nonsmooth component has an explicit max structure that may not easy to compute its proximal operator. The feasible region is required to be compact for convergence.

In this paper, we propose an adaptive smoothing  mini-batch stochastic accelerated gradient (AdaSMSAG) method for solving stochastic optimization problems \eqref{orip} with a general nonsmooth convex component $h$. We do not assume that $h$ has the max structure in \eqref{h}, or its proximal operator is easily obtained. The feasible region does not need to be compact. We can adopt any smoothing function $\tilde h_{\mu}(x)$ for $h(x)$ that satisfies Definition \ref{definition 1} below to obtain the smooth problem
\begin{eqnarray}\label{smoothp}
\min_{x\in X}\left\{\tilde \psi_{\mu}(x) := f(x) + \tilde h_{\mu}(x) \right\}.
\end{eqnarray}
We adopt the updating rule of the smoothing parameter similar as that in \cite{Quoc} to fasten the computational speed.

The AdaSMSAG method proposed in this paper is suitable for the stochastic setting. We show the convergence of the AdaSMSAG method by
\begin{eqnarray*}
E\left[\psi(y_{N})-\psi^*\right]\leq O\left(\frac{\ln (N+c) }{N}\right),
\end{eqnarray*}
where $c>0$ is a constant, and $N\geq 1$ is the maximum iteration number. Because
$
\ln (N+c) < \sqrt{N+c},
$
the order of the worst-case iteration complexity of the AdaSMSAG method is better than $O\left(\frac{1}{\epsilon^2}\right)$ that is considered to be the optimal order of the worst-case iteration complexity as in \cite{Nemirovski,NemirovskiY}.  Similar as for the deterministic setting \cite{Nesterov}, the optimal order $O\left(\frac{1}{\epsilon^2}\right)$ of the worst-case iteration complexity  is valid only for the black-box oracle model of the objective function. In practice, we always know something about the structure of the underlying objective functions. The proper use of the structure of the problem by smoothing approximation adopted in this paper does help in improving the efficiency to find an $\epsilon$-approximate solution.

At the $k$-th iterate, although $f$ and $\tilde{h}_{\mu_k}(x)$ are Lipschitz continuously differentiable, we assume that only the noisy gradients of $f$ and $\tilde{h}_{\mu_k}(x)$ are available via subsequent calls to a stochastic first-order oracle ($\mathcal{SFO}$). That is,  when we solve the smooth problem \eqref{smoothp} by an iterative method,
at the $k$-th iterate, $k\geq 1$, for the input $x_k\in X$, the $\mathcal{SFO}$ would output a stochastic gradient $\nabla\tilde{\Psi}_{\mu_k}(x_k, \xi_k)$, where $\xi_k$ is a random vector whose probability distribution is supported on $\Xi\subseteq\mathds{R}^d$.
Throughout the paper, we make the following assumptions for the functions $\tilde{\Psi}_{\mu_k}(\cdot, \xi_k)$.
\begin{assumption}\label{assumption 1}
For any fixed $\mu_k>0$ and $x_k\in X$, we have
\begin{eqnarray*}
&a)&E\left[\nabla\tilde{\Psi}_{\mu_k}(x_k, \xi_k)\right]=\nabla\tilde{\psi}_{\mu_k}(x_k),\\
&b)&E\left[\|\nabla\tilde{\Psi}_{\mu_k}(x_k, \xi_k)-\nabla\tilde{\psi}_{\mu_k}(x_k)\|^2\right]\leq \sigma^2,
\end{eqnarray*}
where $\sigma>0$ is a constant, and the expectation $E$ is taken with respect to the random vector $\xi_k\in \Xi$.
\end{assumption}

The remaining of this paper is organized as follows. In Section 2, we briefly review some basic concepts and results relating to the smoothing function \cite{bian} that will be used in our paper.
In Section 3, we develop an adaptive smoothing mini-batch stochastic accelerated gradient (AdaSMSAG) method. We show the convergence, as well as the order of the worst-case iteration complexity of the proposed method that is better than the state-of-the-art stochastic approximation methods for stochastic nonsmooth convex optimization. Numerical experiments on a risk management in portfolio optimization and a family of Wasserstein distributionally robust support vector machine (DRSVM) problems with real data are given in Section 4 to demonstrate the efficiency of our proposed method.

Throughout the paper, we use the following notation. We denote by $\mathds{Z_+}$ the positive integer set, by $\|x\|$ the Euclidean norm of a vector $x$ in the Euclidean space $\mathds{R}^n$. Given $x=\left(x_{1}, \ldots, x_{n}\right)^T \in \mathds{R}^{n}$ and $y= \left(y_{1}, \ldots, y_{n}\right)^T \in \mathds{R}^{n}$, the inner product of $x$ and $y$ is defined by $\langle x, y\rangle:=\sum\limits_{i=1}^{n} x_{i} y_{i}$. The notation $\lceil r\rceil$ stands for the smallest integer greater than or equal to $r\in \mathds{R}$. For a number $r\in \mathds{R}$ we denote $[r]_+ := \max\{r, 0\}$.

\section{Preliminaries}
In this section, we review some basic concepts and results that will be used later. We consider a class of smoothing functions with the following definition.

\begin{definition}\label{definition 1}
We call $\tilde \phi_{\mu}$ with $\mu \in(0, \bar{\mu}]$ a smoothing function of the convex function $\phi$ in the feasible region, if $\tilde{\phi}_{\mu}(\cdot)$ is continuously differentiable in $X$ for any fixed $\mu>0$ and satisfies the following conditions:\\
$(i)$   $\lim _{z \rightarrow x, \mu \downarrow 0} \tilde{\phi}_{\mu}(z) = \phi(x),~~\forall x \in X$;\\
$(ii)$  (convexity) $\tilde{\phi}_{\mu}(x)$ is convex with respect to $x$ in $X$ for any fixed $\mu>0$;\\
%$(iii)$ (gradient consistency) $\left\{\lim _{z \rightarrow x, \mu \downarrow 0} \nabla \tilde{h}_{\mu}(z)\right\} \subseteq \partial h(x), %\forall x \in X$;\\
$(iii)$  (Lipschitz continuity with respect to $\mu$) there exists a positive constant $\kappa$ such that
\begin{eqnarray}\label{hmurelation}
\lvert \tilde{\phi}_{\mu_2}(x)-\tilde{\phi}_{\mu_1}(x)\rvert \leq \kappa\lvert\mu_{1}-\mu_{2}\rvert,~~\forall x \in X,~\mu_{1}, \mu_{2} \in(0, \bar{\mu}];
\end{eqnarray}
$(iv)$  (Lipschitz continuity with respect to $x$) there exists a constant $L>0$ such that for any $\mu \in(0, \bar{\mu}], \nabla \tilde{\phi}_{\mu}(\cdot)$ is Lipschitz continuous on $X$ with Lipschitz constant $\frac{L}{\mu}$.
\end{definition}

Note that the smoothing function in Definition 3.1 of \cite{bian} requires all the above conditions in Definition \ref{definition 1}, besides another gradient consistency property.
Definition \ref{definition 1} $(iii)$ implies
\begin{eqnarray}\label{hbound}
\lvert\tilde{\phi}_{\mu}(x)-\phi(x)\rvert \leq \kappa \mu, \quad \forall x \in X,~\mu \in(0, \bar{\mu}].
\end{eqnarray}
Many smoothing functions satisfy Definition 1.

\begin{example}
Chen and Mangasarian construct a class of smooth approximations of the function $(t)_{+}$ by convolution \cite{chunhua} as follows. Let $\rho: \mathds{R} \rightarrow \mathds{R}_{+}$ be a piecewise continuous density function satisfying
$\rho(s)=\rho(-s)$. It is easy to see that
\begin{eqnarray*}
\tilde{\phi}_{\mu}(t):=\int_{-\infty}^{\infty}(t-\mu s)_{+} \rho(s) ds
\end{eqnarray*}
is a smoothing function of $(t)_+$. According to the inequality (22) of \cite{Chen1}, we have
$${\kappa} = {\int_{-\infty}^{\infty}} \vert s\vert \rho(s) ds<\infty$$
 in Definition \ref{definition 1} $(iv)$. By choosing different density functions, we can obtain different smoothing function, e.g.,
the uniform smoothing function, the Neural Networks smoothing function, and the CHKS (Chen-Harker-Kanzow-Smale) smoothing function; see \cite{Chen1} for details.
\end{example}

%\begin{example}
%Nesterov \cite{Nesterov} construct a class of smooth approximations of the function \eqref{h} as
%\begin{eqnarray*}
%\tilde{h}_{\mu}(x)=\max\limits_{u\in U}\left\{Ax, u\rangle - Q(u)-\mu d(u)\right\},
%\end{eqnarray*}
%where $d(u)$ is $\sigma_d$-strongly convex on the bounded closed convex set $U$.
%Then we have in Definition \ref{definition 1} $(iv)$
%\begin{eqnarray*}
%\kappa =\lvert\nabla_{\mu}\tilde{h}_{\mu}(x)\rvert=\lvert\max\limits_{u\in U}\left\{-d(u)\right\}\rvert.
%\end{eqnarray*}
%\end{example}
Using Definition \ref{definition 1}, we then find the objective function $\tilde{\psi}_{\mu_k}(x)$ for the smooth problem in \eqref{smoothp} with $\mu = \mu_k$ is $L_k$-smooth on $X$, i.e., $\tilde{\psi}_{\mu_k}(x)$ is Lipschitz continuously differentiable with Lipschitz constant $L_k=L_f+L_{\tilde{h}_{\mu_k}}$ on $X$, where $L_f$ is the Lipschitz constant of $f$. According to Theorem 5.12 of \cite{Beck}, we get $L_f\geq\|\nabla^2f(x)\|=\lambda_{\max}(\nabla^2f(x))$ for any $x\in X$, where $\lambda_{\max}(\nabla^2f(x))$ means the maximal eigenvalue of the Hessian matrix of $f(x)$. Thus
\begin{eqnarray}\label{psiLsmoo}
\tilde{\psi}_{\mu_k}(y)\leq \tilde{\psi}_{\mu_k}(x)+\langle \nabla \tilde{\psi}_{\mu_k}(x),y-x\rangle+\frac{L_k }{2}\|y-x\|^2.
\end{eqnarray}

%Given a vector $g\in \mathds{R}^n$, we define the generalized projected gradient of $\psi$ at $x \in X$ as
%\begin{eqnarray}\label{pg}
%P_X(x,g,\gamma)=\frac{1}{\gamma}\left(x-x'\right),
%\end{eqnarray}
%where $x'$ is given by
%\begin{eqnarray}\label{x'}
%x'=\arg\min\limits_{y\in X}\left\{\langle g,y\rangle+\frac{1}{2\gamma}\|y-x\|^2\right\}.
%\end{eqnarray}

\section{Adaptive smoothing mini-batch stochastic accelerated gradient (AdaSMSAG) method}

In this section, we develop an adaptive smoothing mini-batch stochastic accelerated gradient (AdaSMSAG) method for \eqref{orip}. We will show the convergence as well as the order of worst-case iteration complexity, as well as the mini-batch sizes of the AdaSMSAG method.

At the $k$-th iterate of the AdaSMSAG method, we randomly choose a mini-bath samples $\xi^{[m_k]}:=\{\xi_{k,1},\ldots,\xi_{k,m_k}\}$ of the random vector $\xi\in \Xi$, where $m_k$ is the batch size. We denote by $\xi_{[k]}:=\{\xi^{[m_1]},\ldots,\xi^{[m_{k}]}\}$ the history of mini-batch samples from the first iterate up to the $k$-th iterate.

For any $k\geq 1$, we denote the mini-batch stochastic gradient $\nabla \tilde{\Psi}_{\mu_k}^k$ by
\begin{eqnarray}\label{Psi}
\nabla \tilde{\Psi}_{\mu_k}^k=\frac{1}{m_{k}}\sum\limits_{j=1}^{m_{k}}\nabla \tilde{\Psi}_{\mu_k}(x_k, \xi_{k,j}).
\end{eqnarray}
For ease of notation, we denote
\begin{eqnarray}
&&\delta_{\mu_k}(x_k)=\nabla\tilde{\Psi}_{\mu_k}^k-\nabla\tilde{\psi}_{\mu_k}(x_k),\label{deltamuk}\\
&&\delta_{\mu_k}^i(x_k)=\nabla\tilde{\Psi}_{\mu_k}(x_k,\xi_{k,i})-\nabla\tilde{\psi}_{\mu_k}(x_k),\ i=1,\cdots,m_k, \label{deltamuki}\\
&&S_{\mu_k}^j(x_{k})=\sum\limits_{i=1}^{j}\delta_{\mu_k}^i(x_k),\ j =1,\cdots,m_k, \label{Smuk}
\end{eqnarray}
where $\delta_{\mu_k}(x_k)=\frac{1}{m_{k}}\sum\limits_{i=1}^{m_{k}}\delta_{\mu_k}^i(x_k)$.

The following lemma addresses the relations of $\nabla \tilde{\Psi}_{\mu_k}^k$ to that of the original gradient $\nabla \tilde{\psi}_{\mu_k}(x_k)$, which can be shown without difficulty using the arguments similar as in \cite{Ghadimi}.

\begin{lemma}\label{lemma3.1} Under Assumption \ref{assumption 1}, we have for any $k\geq1$ and $\mu_k>0$,
\begin{eqnarray*}
&(a)& E\left[\nabla \tilde{\Psi}_{\mu_k}^k\right] = E\left[\nabla \tilde{\psi}_{\mu_k}(x_k)\right] ,\\
&(b)& E\left[\left\| \nabla \tilde{\Psi}_{\mu_k}^k - \tilde{\psi}_{\mu_k}(x_k) \right\|^2\right] \leq \frac{\sigma^2}{m_k},
\end{eqnarray*}
where the expectation is taken with respect to the history of mini-batch samples $\xi_{[k]}$.
\end{lemma}

%\subsection{Minimization of stochastic smooth functions}
%In this subsection, we consider problem \eqref{orip} by the AdaSMSAG method.
Now we give the AdaSMSAG method as follows.

\begin{algorithm}
\caption{Adaptive smoothing mini-batch stochastic accelerated gradient (AdaSMSAG) method}\label{algo1}
\begin{algorithmic}[1]
\Require the iteration limit $N>0$, the batch sizes $\{m_k\}_{k\geq 1}$ with $m_k \geq 1$, $\{\alpha_k\}_{k\geq 1}$ with $0<\alpha_k< 1$, $c>0, \mu_0>0$, $\{\beta_k\}_{k\geq 1}$ with $\beta_k>0$, and $\{\theta_k\}_{k\geq 1}$ with $\theta_k>0$, $z_{0}=y_{0}$.
%\State $k \Leftarrow 1$
\For{$k=1,\ldots, N$}\label{algln2}
        \State $x_{k}=\alpha_{k}z_{k-1}+(1-\alpha_{k})y_{k-1}$
        \State $\mu_{k}=\frac{c\mu_0}{k+c}$
        \State Call the $\mathcal{SFO}$ $m_k$ times to obtain $\nabla \tilde{\Psi}_{\mu_k}(x_{k}, \xi_{k,i})$, $i=1,\cdots, m_{k}$.
        \State Set $\nabla \tilde{\Psi}_{\mu_k}^{k}=\frac{1}{m_{k}}\sum\limits_{i=1}^{m_{k}}\nabla \tilde{\Psi}_{\mu_k}(x_{k}, \xi_{k,i})$, and compute
        \State $y_{k}=\arg\min\limits_{y\in X}\{\langle \nabla \tilde{\Psi}_{\mu_k}^{k},y-x_{k}\rangle+\frac{\beta_k}{2}\|y-x_{k}\|^2\}$
        \State $z_{k}=\arg\min\limits_{x\in X}\{\langle\nabla\tilde{\Psi}_{\mu_k}^{k},x-x_{k}\rangle+\frac{\theta_{k}}{2}\|x-z_{k-1}\|^2\}$
%\Else
\EndFor
\State {\bf{output}}: $y_{N}$.
\end{algorithmic}
\end{algorithm}
\bigskip

We have the following simple observations in Lemmas 2 and 3.

\begin{lemma}\label{lemma3.2}
If $\mu_k$ is monotonically decreasing with $k$,  we have for any $k\geq 1$ and $x, y_k\in X$,
\begin{eqnarray}
\psi(y_k)-\psi(x)&\leq& \tilde{\psi}_{\mu_k}(y_k)-\tilde{\psi}_{\mu_k}(x)+2\kappa\mu_k,\label{psibound}\\
\tilde{\psi}_{\mu_{k+1}}(y_k)-\tilde{\psi}_{\mu_{k+1}}(x)&\leq&\tilde{\psi}_{\mu_{k}}(y_k)-\tilde{\psi}_{\mu_{k}}(x)+2\kappa\left(\mu_k-\mu_{k+1}\right).\label{psimurelation}
\end{eqnarray}
\end{lemma}
\begin{proof}
For any fixed $x, y_k\in X$ and $\mu_k>0$, we have
\begin{eqnarray*}
\psi(y_k)-\psi(x)&\overset{\eqref{orip}}{=}&\left[f(y_k)+h(y_k)\right]-\left[f(x)+h(x)\right]\\
&\overset{\eqref{hbound}}{\leq}&\left[f(y_k)+\tilde{h}_{\mu_k}(y_k)\right]-\left[f(x)+\tilde{h}_{\mu_k}(x)\right]+2\kappa\mu_k\\
&\overset{\eqref{smoothp}}{=}&\tilde{\psi}_{\mu_k}(y_k)-\tilde{\psi}_{\mu_k}(x)+2\kappa\mu_k.
\end{eqnarray*}

For the inequality \eqref{psimurelation}, in view of $\mu_{k+1}\leq\mu_k$ we find that
\begin{eqnarray*}
\tilde{\psi}_{\mu_{k+1}}(y_k)-\tilde{\psi}_{\mu_{k+1}}(x)
&\overset{\eqref{smoothp}}{=}&f(y_k)+\tilde{h}_{\mu_{k+1}}(y_k)-\left[f(x)+\tilde{h}_{\mu_{k+1}}(x)\right]\\
&\overset{\eqref{hmurelation}}{\leq}&f(y_k)+\tilde{h}_{\mu_{k}}(y_k)+\kappa\left(\mu_k-\mu_{k+1}\right)\\
&&~~~-\left[f(x)+\tilde{h}_{\mu_{k}}(x)-\kappa\left(\mu_k-\mu_{k+1}\right)\right]\\
&\overset{\eqref{smoothp}}{=}&\tilde{\psi}_{\mu_{k}}(y_k)-\tilde{\psi}_{\mu_{k}}(x)+2\kappa\left(\mu_k-\mu_{k+1}\right).
\end{eqnarray*}
\end{proof}

%\begin{lemma}\label{lemma3.3}
%Let $y_{k}$ and $z_{k}$ be obtained by Algorithm \ref{algo1}. By the generalized projected gradient in \eqref{pg}-\eqref{x'}, we have
%\begin{eqnarray}
%x_{k}-y_{k}=\frac{P_X(x_{k},\nabla\tilde{\Psi}_{\mu_{k}}^{k},\gamma_{y_{k}})}{\beta_{k}},~~~~~~~~~~~~~\label{yk+1}\\
%z_{k-1}-z_{k}=\frac{P_X(z_{k-1},\nabla\tilde{\Psi}_{\mu_{k}}^{k},\gamma_{z_{k}})}{\theta_{k}},~~~~~~~~~~~\label{zk+1}\\
%\langle\nabla\tilde{\Psi}_{\mu_{k}}^{k},P_X(x_{k},\nabla\tilde{\Psi}_{\mu_{k}}^{k},\gamma_{y_{k}})\rangle\geq \|P_X(x_{k},\nabla\tilde{\Psi}_{\mu_{k}}^{k},\gamma_{y_{k}})\|^2.\label{gpg}
%\end{eqnarray}
%\end{lemma}
%\begin{proof}
%In view of the definitions of $y_{k}$, $z_{k}$ and the generalized projected gradient, we have
%\begin{eqnarray*}
%P_X(x_{k},\nabla\tilde{\Psi}_{\mu_{k}}^{k},\gamma_{y_{k}})&=&\beta_{k}(x_{k}-y_{k})\\
%P_X(z_{k-1},\nabla\tilde{\Psi}_{\mu_{k}}^{k},\gamma_{z_{k}})&=&\theta_{k}(z_{k-1}-z_{k}).
%\end{eqnarray*}
%Letting $g=\nabla\tilde{\Psi}_{\mu_{k}}^{k}$, $V(y,x) = 1/2\|y-x\|^2$ with convexity parameter $\alpha = 1$, and $h(x) = 0$ in (2.3) of \cite{Ghadimi1}, we have by Lemma 1 of \cite{Ghadimi1} that
%\begin{eqnarray*}
%\langle\nabla\tilde{\Psi}_{\mu_{k}}^{k},P_X(x_{k},\nabla\tilde{\Psi}_{\mu_{k}}^{k},\gamma_{y_{k}})\rangle\geq \|P_X(x_{k},\nabla\tilde{\Psi}_{\mu_{k}}^{k},\gamma_{y_{k}})\|^2.
%\end{eqnarray*}
%\end{proof}

Below we develop Lemmas \ref{lemma3.41} and \ref{lemma3.4} which are important for developing the convergence results.
\begin{lemma}\label{lemma3.41} (Lemma 2 of \cite{Ghadimi1})
Let the convex function $p: X \rightarrow \mathds{R}$, the points $\tilde{x}, \tilde{y} \in X$ and the scalars $\mu_{1}, \mu_{2} \geq 0$ be given. Let $\omega: X \rightarrow \mathds{R}$ be a differentiable convex function and $V(x, z)$ be the Bregman's distance. If
\begin{eqnarray*}
u^{*} \in \arg\min\limits_{u \in X}\left\{p(u)+\mu_{1} V(\tilde{x}, u)+\mu_{2} V(\tilde{y}, u)\right\},
\end{eqnarray*}
then for any $u \in X$, we have
\begin{eqnarray*}
& & p\left(u^{*}\right)+\mu_{1} V\left(\tilde{x}, u^{*}\right)+\mu_{2} V\left(\tilde{y}, u^{*}\right) \\
& & \quad \quad \leq  p(u)+\mu_{1} V(\tilde{x}, u)+\mu_{2} V(\tilde{y}, u)-\left(\mu_{1}+\mu_{2}\right) V\left(u^{*}, u\right) .
\end{eqnarray*}
\end{lemma}

\begin{lemma}\label{lemma3.4}
Let $x^*$ be an arbitrary optimal solution of \eqref{orip}. Denote $D_{k-1}^2=\frac{\|x^*-z_{k-1}\|^2}{2}$. Assume that
\begin{eqnarray}\label{assum2}
\alpha_{k}\beta_{k}-\theta_k\leq 0.
\end{eqnarray}
 We have for any $k\geq 1$,
\begin{eqnarray}\label{pu}
\langle\nabla\tilde{\Psi}_{\mu_{k}}^{k},y_{k}-\alpha_{k}x^*-(1-\alpha_{k})y_{k-1}\rangle
\leq -\frac{\beta_{k}}{2}\|y_{k}-x_{k}\|^2
+\alpha_{k}\theta_{k}\left(D_{k-1}^2
-D_{k}^2\right).
\end{eqnarray}
\end{lemma}
\begin{proof}
 Letting $p(u)=\langle\nabla\tilde{\Psi}_{\mu_{k}}^{k},u-x_{k}\rangle$, $u=x^*$, $\widetilde{x}=z_{k-1}$, $u^*=z_{k}$, $\mu_1=\theta_{k}$, and $\mu_2=0$ In Lemma \ref{lemma3.41}, we have by the definition of $z_{k}$ that
\begin{eqnarray}\label{z}
\langle\nabla\tilde{\Psi}_{\mu_{k}}^{k},z_{k}-x^*\rangle\leq\theta_{k}\left(D_{k-1}^2
-D_{k}^2\right)-\frac{\theta_{k}\|z_{k}-z_{k-1}\|^2}{2}.
\end{eqnarray}
Similarly, letting $p(u)=\langle\nabla\tilde{\Psi}_{\mu_{k}}^{k},u-x_{k}\rangle$, $u=x\in X$, $\widetilde{x}=x_{k}$, $u^*=y_{k}$, $\mu_1=\beta_{k}$, and $\mu_2=0$ in Lemma \ref{lemma3.41}, we have by the definition of $y_{k}$ thfat
\begin{eqnarray*}
\langle\nabla\tilde{\Psi}_{\mu_{k}}^{k},y_{k}-x\rangle
\leq\frac{\beta_{k}}{2}\left(\|x_{k}-x\|^2-\|y_{k}-x\|^2-\|y_{k}-x_{k}\|^2\right).
\end{eqnarray*}
By choosing $x=\alpha_{k}z_{k}+(1-\alpha_{k})y_{k-1}\in X$ in the above inequality, we obtain
\begin{eqnarray*}
&&\langle\nabla\tilde{\Psi}_{\mu_{k}}^{k},y_{k}-\alpha_{k}z_{k}-(1-\alpha_{k})y_{k-1}\rangle\\
&&~~~\leq\frac{\beta_{k}}{2}\left(\|x_{k}-\alpha_{k}z_{k}-(1-\alpha_{k})y_{k-1}\|^2-\|y_{k}-x_{k}\|^2\right)\\
&&~~~=\frac{\beta_{k}}{2}\left(\alpha_{k}^2\|z_{k-1}-z_{k}\|^2-\|y_{k}-x_{k}\|^2\right),
\end{eqnarray*}
where the last equality follows from the definition of $x_k$ in Algorithm \ref{algo1}.
Multiplying \eqref{z} by $\alpha_{k}$ and then summing it to the above inequality, we have
\begin{eqnarray*}
&&\langle\nabla\tilde{\Psi}_{\mu_{k}}^{k},y_{k}-\alpha_{k}x^*-(1-\alpha_{k})y_{k-1}\rangle\\
&&~~~\leq\frac{\alpha_{k}(\alpha_{k}\beta_{k}-\theta_k)}{2}\|z_{k-1}-z_{k}\|^2
-\frac{\beta_{k}}{2}\|y_{k}-x_{k}\|^2
+\alpha_{k}\theta_{k}\left(D_{k-1}^2
-D_{k}^2\right)\\
&&~~~\leq -\frac{\beta_{k}}{2}\|y_{k}-x_{k}\|^2
+\alpha_{k}\theta_{k}\left(D_{k-1}^2
-D_{k}^2\right),
\end{eqnarray*}
since $\alpha_{k}\beta_{k}-\theta_k\leq 0$ according to (\ref{assum2}).
\end{proof}

For ease of notation, we define
\begin{eqnarray}
\Delta_{\mu_{k-1}}=\tilde{\psi}_{\mu_{k-1}}(y_{k-1})-\tilde{\psi}_{\mu_{k-1}}(x^*),~~~~~~~~~~~~~\label{Deltamuk}\\
\Gamma_{k}=\Delta_{\mu_{k}}-(1-\alpha_{k})\Delta_{\mu_{k-1}}-2\kappa(1-\alpha_{k})\left(\mu_{k-1}-\mu_{k}\right).\label{Gammak+1}
\end{eqnarray}

\begin{lemma}\label{lemma3.5}
Let $x^*$ be an arbitrary optimal solution of \eqref{orip}. Assume that
\begin{eqnarray}\label{assum3}
\alpha_{k}\beta_{k}-\theta_k\leq 0\quad \mbox{and}\quad \beta_{k}-L_{k}>0.
\end{eqnarray}
For any $k\geq 1$ we have
\begin{eqnarray}\label{theorem3.1eq}
\Gamma_{k}\leq\alpha_{k}\theta_{k}\left(D_{k-1}^2
-D_{k}^2\right)+\frac{\|\delta_{\mu_{k}}(x_{k})\|^2}{2(\beta_{k}-L_{k})}
+\alpha_{k}\langle\delta_{\mu_{k}}(x_{k}),x^*-z_{k-1}\rangle.
\end{eqnarray}
\end{lemma}
\begin{proof}
Based on \eqref{psimurelation}, we have
\begin{eqnarray}\label{psimuykminusx}
&&(1-\alpha_{k})\tilde{\psi}_{\mu_{k}}(y_{k-1})-\tilde{\psi}_{\mu_{k}}(x^*)\nonumber\\
&&~~~=~~(1-\alpha_{k})\left[\tilde{\psi}_{\mu_{k}}(y_{k-1})-\tilde{\psi}_{\mu_{k}}(x^*)\right]-\alpha_{k}\tilde{\psi}_{\mu_{k}}(x^*)\nonumber\\
&&~~\overset{\eqref{psimurelation}}{\leq}(1-\alpha_{k})\left[\tilde{\psi}_{\mu_{k-1}}(y_{k-1})-\tilde{\psi}_{\mu_{k-1}}(x^*)\right]
+2\kappa(1-\alpha_{k})\left(\mu_{k-1}-\mu_{k}\right)\nonumber\\
&&~~~~~~~~~~~~-\alpha_{k}\tilde{\psi}_{\mu_{k}}(x^*)\nonumber\\
&&~~\overset{\eqref{Deltamuk}}{\leq}(1-\alpha_{k})\Delta_{\mu_{k-1}}+2\kappa(1-\alpha_{k})\left(\mu_{k-1}-\mu_{k}\right)-\alpha_{k}\tilde{\psi}_{\mu_{k}}(x^*).
\end{eqnarray}
By the convexity of $\tilde{\psi}_{\mu_{k}}$, we have
\begin{eqnarray}\label{convex1}
&&\tilde{\psi}_{\mu_{k}}(x_{k})\leq \tilde{\psi}_{\mu_{k}}(x)-\langle\nabla\tilde{\psi}_{\mu_{k}}(x_{k}),x-x_{k}\rangle,~~\forall x\in X.
\end{eqnarray}
In view of $\tilde{\psi}_{\mu_{k}}$ is an $L_{k}$-smooth function over a given convex set $X$, we have
\begin{eqnarray*}
&&\tilde{\psi}_{\mu_{k}}(y_{k})\nonumber\\
&&~\overset{\eqref{psiLsmoo}}{\leq}\tilde{\psi}_{\mu_{k}}(x_{k})+\langle\nabla\tilde{\psi}_{\mu_{k}}(x_{k}),y_{k}-x_{k}\rangle+\frac{L_{k}}{2}\|y_{k}-x_{k}\|^2\nonumber\\
&&~~~=(1-\alpha_{k})\tilde{\psi}_{\mu_{k}}(x_{k})+\alpha_{k}\tilde{\psi}_{\mu_{k}}(x_{k})+\langle\nabla\tilde{\psi}_{\mu_{k}}(x_{k}),y_{k}-x_{k}\rangle+\frac{L_{k}}{2}\|y_{k}-x_{k}\|^2\nonumber\\
&&~\overset{\eqref{convex1}}{\leq}(1-\alpha_{k})\tilde{\psi}_{\mu_{k}}(y_{k-1})
-(1-\alpha_{k})\langle\nabla\tilde{\psi}_{\mu_{k}}(x_{k}),y_{k-1}-x_{k}\rangle\nonumber\\
&&~~~~~~~~~+\alpha_{k}\tilde{\psi}_{\mu_{k}}(x^*)
-\alpha_{k}\langle\nabla\tilde{\psi}_{\mu_{k}}(x_{k}),x^*-x_{k}\rangle+\langle\nabla\tilde{\psi}_{\mu_{k}}(x_{k}),y_{k}-x_{k}\rangle\nonumber\\
&&~~~~~~~~~~~~+\frac{L_{k}}{2}\|y_{k}-x_{k}\|^2.
\end{eqnarray*}

Subtracting $\tilde{\psi}_{\mu_{k}}(x^*)+(1-\alpha_{k})\Delta_{\mu_{k-1}}+2\kappa(1-\alpha_{k})\left(\mu_{k-1}-\mu_{k}\right)$ from both sides of the above inequality, and according to \eqref{psimuykminusx}, we get
\begin{eqnarray*}
\Gamma_{k}&\leq&(1-\alpha_{k})\tilde{\psi}_{\mu_{k}}(y_{k-1})-\tilde{\psi}_{\mu_{k}}(x^*)-(1-\alpha_{k})\Delta_{\mu_{k-1}}\\
&&~~-2\kappa(1-\alpha_{k})\left(\mu_{k-1}-\mu_{k}\right)
-(1-\alpha_{k})\langle\nabla\tilde{\psi}_{\mu_{k}}(x_{k}),y_{k-1}-x_{k}\rangle\\
&&~~~~~+\alpha_{k}\tilde{\psi}_{\mu_{k}}(x^*)
-\alpha_{k}\langle\nabla\tilde{\psi}_{\mu_{k}}(x_{k}),x^*-x_{k}\rangle+\langle\nabla\tilde{\psi}_{\mu_{k}}(x_{k}),y_{k}-x_{k}\rangle\nonumber\\
&&~~~~~~~~~+\frac{L_{k}}{2}\|y_{k}-x_{k}\|^2\\
&\overset{\eqref{psimuykminusx}}{\leq}&\langle\nabla\tilde{\psi}_{\mu_{k}}(x_{k}),y_{k}-\alpha_{k}x^*-(1-\alpha_{k})y_{k-1}\rangle+\frac{L_{k}}{2}\|y_{k}-x_{k}\|^2\\
&\overset{\eqref{deltamuk}}{=}&\langle\nabla\tilde{\Psi}_{\mu_{k}}^{k}-\delta_{\mu_{k}}(x_{k}),y_{k}-\alpha_{k}x^*-(1-\alpha_{k})y_{k-1}\rangle
+\frac{L_{k}}{2}\|y_{k}-x_{k}\|^2\\
&\overset{\eqref{pu}}{\leq}&\alpha_{k}\theta_{k}\left(D_{k-1}^2
-D_{k}^2\right)-\frac{\beta_{k}-L_{k}}{2}\|y_{k}-x_{k}\|^2
+\langle\delta_{\mu_{k}}(x_{k}),x_{k}-y_{k}\rangle\\
&&~~~~~-\langle\delta_{\mu_{k}}(x_{k}),x_{k}-\alpha_{k}x^*-(1-\alpha_{k})y_{k-1}\rangle\\
&\leq&\alpha_{k}\theta_{k}\left(D_{k-1}^2
-D_{k}^2\right)+\frac{\|\delta_{\mu_{k}}(x_{k})\|^2}{2(\beta_{k}-L_{k})}
+\alpha_{k}\langle\delta_{\mu_{k}}(x_{k}),x^*-z_{k-1}\rangle.
\end{eqnarray*}
The last inequality follows from the fact that $at-\frac{bt^2}{2}\leq\frac{a^2}{2b}$ for any $a,t,b>0$, and we choose special $a=\|\delta_{\mu_{k}}(x_{k})\|$, $t=\|x_{k}-y_{k}\|$, $b=\beta_{k}-L_{k}>0$ here.
\end{proof}

\begin{proposition}
Under Assumption \ref{assumption 1}, we have for any $k\geq 1$,
\begin{eqnarray}\label{prop3.1eq}
E\left[\Gamma_{k}\right]
\leq\alpha_{k}\theta_{k}E\left[D_{k-1}^2
-D_{k}^2\right]+\frac{\sigma^2}{2m_k(\beta_{k}-L_{k})},
\end{eqnarray}
where $\Gamma_{k}$ is defined in \eqref{Gammak+1} and the expectation $E$ is taken with respect to $\xi_{[k]}$.
\end{proposition}
\begin{proof}
Based on Lemma \ref{lemma3.1} (a), we have for any $k\geq 1$
\begin{eqnarray}\label{delta0}
&&E\left[\langle\delta_{\mu_{k}}(x_{k}),x^*-z_{k-1}\rangle\right]\nonumber\\
&&~~=~E\left\{E\left[\left\langle\delta_{\mu_{k}}(x_{k}),x^*-z_{k-1}\right\rangle\vert\xi_{[k-1]}\right]\right\}\nonumber\\
&&~~=~E\left\{\left\langle E\left[\delta_{\mu_{k}}(x_{k})\vert\xi_{[k-1]}\right],x^*-z_{k-1}\right\rangle\right\}\nonumber\\
&&~\overset{\eqref{deltamuk}}{=}E\left\{\left\langle E\left[\nabla\tilde{\Psi}_{\mu_{k}}^{k}\vert\xi_{[k-1]}\right] -\nabla\tilde{\psi}_{\mu_{k}}(x_{k}),x^*-z_{k-1}\right\rangle\right\}\nonumber\\
&&~~=~E\left\{\left\langle\nabla\tilde{\psi}_{\mu_{k}}(x_{k})-\nabla\tilde{\psi}_{\mu_{k}}(x_{k}),x^*-z_{k-1}\right\rangle\right\}\nonumber\\
&&~~=~0,
\end{eqnarray}
where the first equality is based on Theorem 3.24 of \cite{Wasserman} and the expectation $E$ is taken with respect to $\xi_{[k]}$.
Taking the expectation on both sides of \eqref{theorem3.1eq} with respect to $\xi_{[k]}$, and using the observations \eqref{delta0}, and Lemma \ref{lemma3.1} (b), we get \eqref{prop3.1eq} as we desired.
\end{proof}

We can choose special sequences $\{\alpha_{k}\}_{k\geq1}, \{\beta_{k}\}_{k\geq1}$ and $\{\theta_{k}\}_{k\geq1}$ such that the requirements \eqref{assum3} in Lemma \ref{lemma3.5} are satisfied as in the following theorem.

\begin{theorem}\label{prop3.2}
We choose in the AdaSMSAG method (Algorithm \ref{algo1}) the parameters
\begin{eqnarray*}
\alpha_{k}=\frac{1}{k},\ \mu_k=\frac{c\mu_0}{k+c},\ c\in \mathds{Z_+},\ \beta_{k}=2L_{k},\
\theta_{k}\equiv\theta=2L_f+\frac{(c+2)L}{c\mu_0},
\end{eqnarray*} and the batch sizes $m_{k} \equiv m> 1$ for $k =1,\ldots, N$. Under Assumption \ref{assumption 1}, we have for any $k\geq 1$,
\begin{eqnarray}\label{prop3.2eq}
E\left[\Delta_{\mu_{k}}\right]
\leq\frac{2L_fD_{0}^2}{k}+\frac{(c+2)LD_{0}^2}{c\mu_0k}+\frac{2\kappa c\mu_0}{k}\ln\left(k+c\right)+\frac{\sigma^2c\mu_0}{2mL}.
\end{eqnarray}
\end{theorem}
\begin{proof}
In view of $\mu_k=\frac{c\mu_0}{k+c}, c\in \mathds{Z_+}$, we find for any $k\geq 1$,
\begin{eqnarray*}
 \alpha_k\beta_k=\frac{2L_{k}}{k}=\frac{2}{k}\left(L_{f}+\frac{L}{\mu_k}\right)
=2 L_{f}\left(\frac{1}{k}+\frac{1}{c^2}+\frac{1}{ck}\right)\le 2L_f+\frac{(c+2)L}{c\mu_0} = \theta_k,
\end{eqnarray*}
and
\begin{eqnarray*}
\beta_k - L_k = L_k >0.
\end{eqnarray*}
Thus the assumptions in (\ref{assum3}) hold.
 Let us denote for simplicity
\begin{eqnarray}
P_{k}:=2\kappa(1-\alpha_{k})\left(\mu_{k-1}-\mu_{k}\right),\quad\mbox{and}\quad
Q_{k}:=\frac{\sigma^2}{2m(\beta_{k}-L_{k})}.\label{PQ}
\end{eqnarray}
By \eqref{prop3.1eq} and \eqref{PQ}, we have for any $k\geq 1$
\begin{eqnarray*}
E\left[ \Delta_{\mu_{k}}\right]+\alpha_{k}\theta  E\left[D_{k}^2\right]
\overset{\eqref{prop3.1eq}}{\leq}(1-\alpha_k)E\left[\Delta_{\mu_{k-1}}\right]+\alpha_{k}\theta E\left[D_{k-1}^2\right]+P_{k}+Q_{k}.
\end{eqnarray*}
Multiplying both sides of this inequality by $\frac{1}{\alpha_{k}\theta}$, we obtain
\begin{eqnarray}\label{prop3.2-1}
\frac{1}{\alpha_{k}\theta}E\left[\Delta_{\mu_{k}}\right]+E\left[D_{k}^2\right]
\leq\frac{1-\alpha_{k}}{\alpha_{k}\theta}E\left[\Delta_{\mu_{k-1}}\right]
+E\left[D_{k-1}^2\right]+\frac{P_{k}+Q_{k}}{\alpha_{k}\theta}.
\end{eqnarray}
Noticing that  $\alpha_{k}=\frac{1}{k}$ for $k = 1,\ldots, N$, we have for any $k\geq 2$,
\begin{eqnarray*}
\frac{1-\alpha_{k}}{\alpha_{k}}=\frac{1}{\alpha_{k-1}}.
\end{eqnarray*}
According to the above equality and \eqref{prop3.2-1}, we have
\begin{eqnarray}\label{prop3.2-2}
\frac{1}{\alpha_{k}\theta}E\left[\Delta_{\mu_{k}}\right]+E\left[D_{k}^2\right]
&\leq&\frac{1}{\alpha_{k-1}\theta}E\left[\Delta_{\mu_{k-1}}\right]
+E\left[D_{k-1}^2\right]+\frac{P_{k}+Q_{k}}{\alpha_{k}\theta}\nonumber\\
&\leq&\ldots\nonumber\\
&\leq&\frac{1-\alpha_{1}}{\alpha_{1}\theta}E\left[\Delta_{\mu_{0}}\right]+D_{0}^2
+\sum\limits_{i=1}^{k}\frac{P_i+Q_i}{\alpha_i\theta}\nonumber\\
&=&D_{0}^2+\sum\limits_{i=1}^{k}\frac{P_i+Q_i}{\alpha_i\theta}.
\end{eqnarray}

Based on the definitions of $P_{k}$, $Q_{k}$, and the elementary inequalities $\sum\limits_{i=1}^{k+c} \frac{1}{i}<1+\ln (k+c)$ and $\sum\limits_{i=1}^{1+c} \frac{1}{i}\geq\frac{3}{2}$, we get
\begin{eqnarray}
\sum\limits_{i=1}^{k}\frac{P_i}{\alpha_i\theta}
&=&\sum\limits_{i=1}^{k}\frac{2\kappa(1-\alpha_i)\left(\mu_{i}-\mu_{i+1}\right)}{\alpha_i\theta}\nonumber\\
&=&\frac{2\kappa c\mu_0}{\theta}\sum\limits_{i=1}^{k} (i-1)\left(\frac{1}{i+c}-\frac{1}{i+c+1}\right)\nonumber\\
&\leq&\frac{2\kappa c\mu_0}{\theta}\sum\limits_{i=2}^{k} \frac{1}{i+c}\nonumber\\
&\leq&\frac{2\kappa c\mu_0}{\theta}\ln \left(k+c\right),\label{Psum}
\end{eqnarray}
and
\begin{eqnarray}
\sum\limits_{i=1}^{k}\frac{Q_i}{\alpha_i\theta}
&=&\sum\limits_{i=1}^{k}\frac{\sigma^2}{2m(\beta_{i}-L_{i})\alpha_i\theta}\nonumber\\
&=&\frac{\sigma^2}{2m\theta}\sum\limits_{i=1}^{k}\frac{i}{L_f+\frac{L}{\mu_{i}}}\nonumber\\
&\leq&\frac{\sigma^2c\mu_0k}{2m\theta L}.\label{Qsum}
\end{eqnarray}
With the above observations \eqref{Psum}, and \eqref{Qsum}, and multiplying \eqref{prop3.2-2} by $\alpha_{k}\theta=\frac{\theta}{k}=\frac{2L_f}{k}+\frac{(c+2)L}{c\mu_0k}$, we obtain
\begin{eqnarray*}
E\left[\Delta_{\mu_{k}}\right]
\leq\frac{2L_fD_{0}^2}{k}+\frac{(c+2)LD_{0}^2}{c\mu_0k}+\frac{2\kappa c\mu_0}{k}\ln\left(k+c\right)+\frac{\sigma^2c\mu_0}{2mL}.
\end{eqnarray*}
\end{proof}

\begin{corollary}\label{corollary3.1}
Let the parameters $\alpha_k$, $\beta_k$, and $\theta_k$ for $k\ge 1$ in the AdaSMSAG method be set as in Theorem \ref{prop3.2}. Let the number of calls to the $\mathcal{SFO}$ at each iteration of the AdaSMSAG method is
\begin{eqnarray}\label{batch}
m=\left\lceil\frac{\sigma^2N}{4\kappa L\ln \left(N+c\right)}\right\rceil.
\end{eqnarray}
Under Assumption \ref{assumption 1}, we have for any $N\geq1$,
\begin{eqnarray}\label{corollary3.1eq}
E\left[\psi(y_{N})-\psi^*\right]\leq\left[2L_fD_{0}^2+\frac{(c+2)LD_{0}^2}{c\mu_0}+2\kappa c\mu_0\right]\frac{1}{N}+\frac{2\kappa c\mu_0}{N}\ln\left(N+c\right).
\end{eqnarray}
\end{corollary}
\begin{proof}
In view of \eqref{psibound}, \eqref{Deltamuk} and \eqref{prop3.2eq}, we have
\begin{eqnarray}
&&E\left[\psi(y_{N})-\psi^*\right]\nonumber\\
&&~\overset{\eqref{psibound}}{\leq}E\left[\tilde{\psi}_{\mu_{N}}(y_{N})-\tilde{\psi}_{\mu_{N}}(x^*)\right]+2\kappa\mu_{N}\nonumber\\
&&~\overset{\eqref{Deltamuk}}{=}E\left[\Delta_{\mu_{N}}\right]+2\kappa\mu_{N}\nonumber\\
&&~\overset{\eqref{prop3.2eq}}{\leq}\frac{2L_fD_{0}^2}{N}+\frac{(c+2)LD_{0}^2}{c\mu_0N}+\frac{2\kappa c\mu_0}{N}\ln\left(N+c\right)+\frac{\sigma^2c\mu_0}{2mL}+\frac{2\kappa c\mu_0}{N+c}\nonumber\\
&&~~\leq\left[2L_fD_{0}^2+\frac{(c+2)LD_{0}^2}{c\mu_0}+2\kappa c\mu_0\right]\frac{1}{N}+\frac{2\kappa c\mu_0}{N}\ln\left(N+c\right)+\frac{\sigma^2c\mu_0}{2mL}.\label{calm}
\end{eqnarray}

By setting in \eqref{calm}
\begin{eqnarray*}
m=\left\lceil\frac{\sigma^2N}{4\kappa L\ln \left(N+c\right)}\right\rceil,
\end{eqnarray*}
we obtain \eqref{corollary3.1eq} as we desired.
\end{proof}
\begin{remark}Since ${\ln}(N+c)< \sqrt{N+c}$, we get
\begin{eqnarray*}
\frac{\ln (N+c)}{N} < \frac{\sqrt{N+c}}{N} = \frac{1}{\sqrt{N+c}} \frac{N+c}{N}.
\end{eqnarray*}
From \eqref{corollary3.1eq}, in order to get an $\epsilon$-approximate solution, we need
\begin{eqnarray*}
E[\psi(y_N)-\psi^*] \le O\left(\frac{\ln (N+c)}{N}\right) \le \epsilon.
\end{eqnarray*}
Thus the order of worst-case iteration complexity for finding an $\epsilon$-approximate solution
is better than $O(\frac{1}{\epsilon^2})$ obtained by the state-of-the-art stochastic approximation methods \cite{Nemirovski,Lan1,Ghadimi1,Ghadimi2}.
\end{remark}

\section{Numerical results}
In this section, we do several numerical experiments to illustrate the efficiency of the AdaSMSAG method by comparing it with several state-of-the-art stochastic approximation methods. We consider a risk management in portfolio optimization, and a family of Wasserstein distributionally robust support vector machine (DRSVM) problems using real data.

We compare our AdaSMSAG method with the following stochastic approximation methods.\\
$\bullet$ The mini-batch stochastic Nesterov's smoothing (MSNS) method \cite{Wang}.\\
$\bullet$ The mini-batch mirror descent stochastic approximation (M-MDSA) method \cite{Nemirovski}.\\
$\bullet$ The randomized stochastic projected gradient (RSPG) method \cite{Ghadimi}.\\
$\bullet$ The two-phase RSPG (2-RSPG) method \cite{Ghadimi}.\\
$\bullet$ The two-phase RSPG variant (2-RSPG-V) method \cite{Ghadimi}.

The MSNS method, the M-MDSA method, the RSPG method, the 2-RSPG method, and the 2-RSPG-V method are state-of-the-art stochastic approximation methods. The MSNS method can solve the first application, i.e., the risk management on portfolio optimization since its nonsmooth component has an explicit max structure. The other four methods can not solve the nonsmooth model defined in \eqref{orip} with guaranteed convergence, because of the relatively complex nonsmooth term that causes the difficulty of obtaining its proximal operator. For comparison, we apply these four methods to solve the smooth counterpart in \eqref{smoothp} with fixed smoothing parameter $\mu\equiv0.001$. In fact, we tried $\mu=0.1,0.01,0.001,0.0001$, and found that $\mu=0.001$ provides the smallest objective value of the original problem \eqref{orip} corresponding to the training data. It's worth noting that the 2-RSPG method and the 2-RSPG-V method include two phases, i.e., the optimization phase and the post-optimization phase. In either the 2-RSPG method, or the 2-RSPG-V method, it generates several candidate outputs in the optimization phase, and the final output is selected from these candidate outputs according to some rules in the post-optimization phase. Since the MDSA method calls one random vector per iteration (the batch size is equal to 1), the total running time is relatively long. We therefore use the mini-batch MDSA method (M-MDSA) to replace the MDSA method in comparison as done in \cite{Ghadimi}. The batch sizes of the M-MDSA method are set to be the same as that for our AdaSMSAG method.

We denote by $N_{total}$ the total number of calls to the stochastic first-order oracle. For each problem, $N_{total}$ are set to be the same for different methods. Since the RSPG, 2-RSPG and 2-RSPG-V methods are randomized SA methods, they stop randomly before up to the maximum number of calls. We then set $N_{total}$ for these three methods to be the maximum number of calls to the stochastic oracle performed in the optimization phase.

Denote by $NS$ the number of the training samples. We use the $NS$ training samples to estimate the parameter, $\sigma^2$. We follow the way of estimating the parameter $\sigma^2$ as in \cite{Ghadimi}. Using the training samples, we compute the stochastic gradients of the objective function $\lceil NS/100\rceil$ times at 100 randomly selected points and then take the average of the variances of the stochastic gradients for each point as an estimation of $\sigma^2$.

Our experiments performed in MATLAB R2018b on a laptop with Windows 10, 1 dualcore 2.4 GHz CPU and 8 GB of RAM. All reported CPU times are in seconds.

\subsection{Risk management in portfolio optimization}

The first application is the risk management on portfolio optimization \cite{Huang1}. Suppose that there are $B$ assets and their random return rates are denoted as $R=\left(R_{1}, R_{2}, \ldots, R_{B}\right)^T$. Assume that for each asset $b=1, \ldots, B$, there are many realizations for its random return rate, with each denoted as $r_{b}$. Denote the random vector $\mathbf{r}=(r_{1},r_{2},\ldots,r_{B})^T$. Let
\begin{eqnarray*}
\mathcal{Z}=\left\{\mathbf{z} \in \mathds{R}^{B}: \sum_{b=1}^{B} z_{b}=1, \mathbf{z} \geq 0\right\}
\end{eqnarray*}
be the set of all feasible portfolio, where $z_{b}$ is the proportion of wealth allocated to asset $b$.

In risk-aware portfolio optimization, the decision maker optimizes an expectation risk measure objective. Conditional Value-at-Risk (CVaR) is one of the most widely investigated risk measures (see \cite{Rock,Krokhmal}). %The CVaR at $\alpha \%$ level is the expected return on the portfolio in the worst $\alpha \%$ of cases, or equivalently, the expected loss on the portfolio in the above $\alpha \%$ of cases.
According to \cite{Huang1}, the $\alpha$-level CVaR objective of portfolio optimization is
\begin{eqnarray}\label{cvar}
\min _{\mathbf{z} \in \mathcal{Z}, \eta \in \mathds{R}} E_{\mathbf{r}}\left[\eta+\frac{1}{1-\alpha}\left(-\sum_{b=1}^{B} r_{b} z_{b}-\eta\right)_{+}\right].
\end{eqnarray}
The objective function in \eqref{cvar} is convex and nonsmooth. In particular, the problem will be large-scale when the number $N$ of realizations for the random vector $\mathbf{r}$ is large. We can then apply our method to tackle this large-scale risk-aware portfolio optimization problem. We employ the Neural Networks smoothing function in \cite{Chen1} and obtain the smoothing problem of \eqref{cvar} as
\begin{eqnarray}\label{smoocvar}
\min _{\mathbf{z} \in \mathcal{Z}, \eta \in \mathds{R}} E_{\mathbf{r}}\left[\eta+\frac{\mu}{1-\alpha}\ln\left(1+e^\frac{- \mathbf{r}^T \mathbf{z}-\eta}{\mu}\right)\right],
\end{eqnarray}
where $\mu$ is the smoothing parameter.

Here, we choose the risk level $\alpha=0.05$. Our data pool consists of the daily returns of exchange-traded funds (ETFs) for B = 20 assets\footnotemark[1]\footnotetext[1]{The abbreviations of these 20 assets are: DIA, EEM, EFA, EWJ, EWZ, FXI, IWM, IYR, OIH, QQQ, SMH, SPY, VWO, XLB, XLE, XLF, XLI, XLK, XLP, XLU.} from October 2007 to October 2017 obtained from the Yahoo Finance website. %So there are 2540 outcomes in total. We impose uniform distribution on the outcomes as the empirical distribution since the true probability distributions on the outcomes of these assets are unknown.
The training set contains the daily returns from October 2007 to September 2017. The out-of-sample performance is based on the return in October 2017.

We compare the out-of-sample performance of the portfolio using the two standard criteria in finance \cite{DeMiguel,Shen}:
(i) Sharpe ratio; (ii) cumulative wealth. These two evaluation metrics represent different focuses on measuring the portfolio performance. The Sharpe ratio (SR) measures the reward-to-risk ratio of a portfolio strategy, which is computed as the portfolio return normalized by its standard deviation:
\begin{eqnarray}\label{SR}
\mathrm{SR}=\frac{\hat{\mu}}{\hat{\sigma}}, \quad \hat{\mu}=\frac{1}{T} \sum_{t=0}^{T-1} \mu_{t}, \quad \hat{\sigma}=\sqrt{\frac{1}{T} \sum_{t=0}^{T-1}\left(\mu_{t}-\hat{\mu}\right)^{2}},
\end{eqnarray}
where $\mu_t$ is the portfolio return from time $t$ to $t+1$, which can be easily calculated as $\mu_t=\omega_{t}^T R_{t}$, $R_{t}$ is the price changes from time $t$ to $t+1$, $\hat{\mu}$ is the mean of portfolio returns and $\hat{\sigma}$ represents the standard deviation of portfolio returns.

The cumulative wealth $(\mathrm{CW})$ of a portfolio measures the total profit from the portfolio strategy across investment periods without considering any risks and costs, which is computed by%By starting with one dollar, the cumulative wealth
\begin{eqnarray}\label{CW}
\mathrm{CW}=\prod_{t=0}^{T-1} \omega_{t}^T R_{t}.
\end{eqnarray}

In each problem, we consider the $\epsilon$-approximate solution with $\epsilon=0.05$ and $0.01$, respectively. For each problem, we run 20 times and record the average results. Over the 20 runs, the average values of the total number $N$ of iterations and the batch sizes $m$, as well as the parameters $L_f, L, \sigma^2$ of the AdaSMSAG method are listed in Table \ref{tab:1}.

\begin{table}[h]
\begin{center}
\begin{minipage}{174pt}
\caption{The average parameters over 20 runs}\label{tab:1}
\begin{tabular}{@{}cccccc@{}}
\toprule
$\epsilon$     & $m$ & $N$ & $L_f$ & $L$  & $\sigma^2$  \\ \midrule
0.05    & 110   & 458 & 1  & 0.25& 1.0081\\ \midrule
0.01    & 520  &2813 & 1  & 0.25& 1.0089  \\
\botrule
\end{tabular}
\end{minipage}
\end{center}
\end{table}

We draw the curves of the average objective values corresponding to the training data v.s. the CPU time for the AdaSMSAG, MSNS, M-MDSA and RSPG methods. We do not include the curves of the 2-RSPG and 2-RSPG-V methods in Fig. \ref{fig:1}, because they have several candidate outputs in optimization phase. Due to scalability, the curves of the AdaSMSAG and the MSNS methods are close at the end. We also provide small graph for each subfigure to see clear the differences of the two methods. We can see that the AdaSMSAG method provides the computed solution with the smallest objective values corresponding to the training data.

\begin{figure}[h]
\subfigure[ $\epsilon=0.05$]
{
\includegraphics[height=4.4cm]{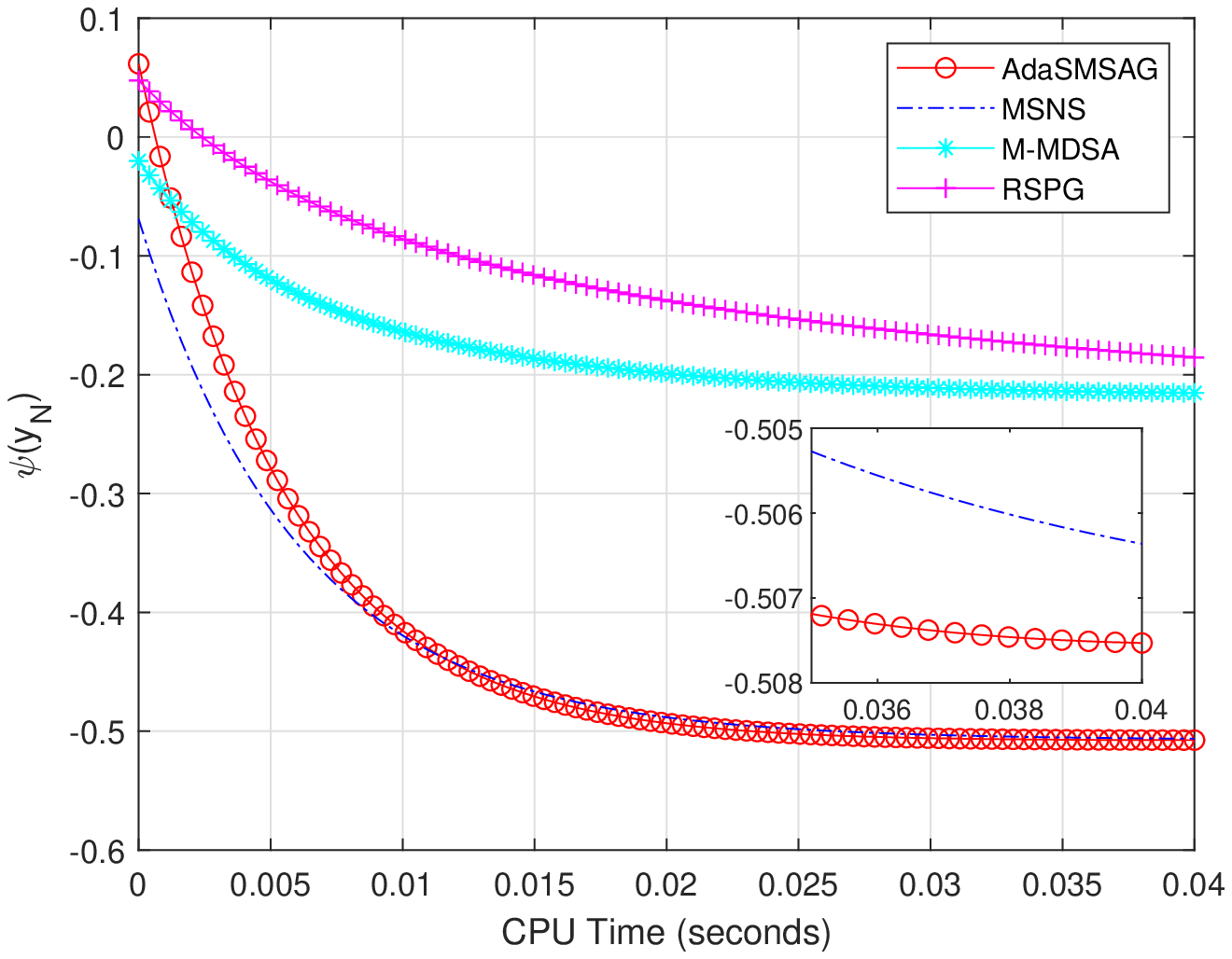}
}\quad
\subfigure[ $\epsilon=0.01$]
{
\includegraphics[height=4.4cm]{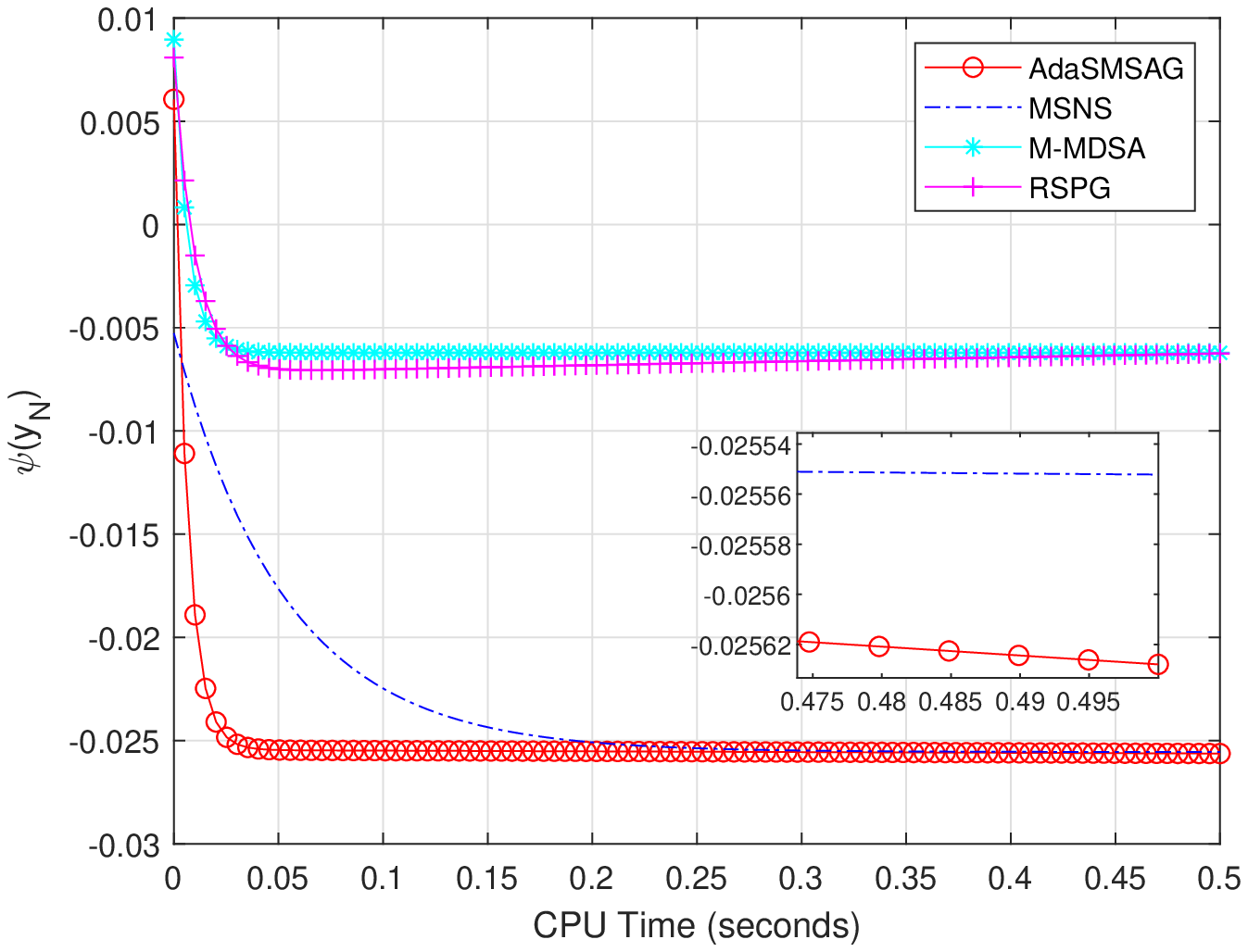}
}
\caption{Average objective values corresponding to the training data v.s. CPU time of 20 runs.}\label{fig:1}
\end{figure}

To see stability, we show in Table \ref{tab:2} the average CPU time, SR and CW, over 20 runs, corresponding to the testing data at the computed solution $\widehat{x}$ by a certain method. From Table \ref{tab:2}, the average CPU time, Sharpe ratio and cumulative wealth of the AdaSMSAG method are the best than those of the other methods in all the cases.

\begin{table}[h]
\begin{center}
\begin{minipage}{225pt}
\caption{The mean SR, CW and CPU time of 20 runs}\label{tab:2}
\begin{tabular}{@{}ccccc@{}}
\toprule
\multirow{2}{*}{$\epsilon$}  & \multirow{2}{*}{ALG.} &Time&SR &CW  \\
 &  & (Avg.)   & (Avg.) & (Avg.) \\\midrule
\multirow{6}{*}{0.05}
&AdaSMSAG&\textbf{0.0751}&\textbf{187.1667}&\textbf{5.0649}\\
&MSNS&0.7778&175.0285&4.8300\\
&M-MDSA&0.0776&170.6928&2.1473\\
&RSPG&3.3500&166.0503&2.4966\\
&2-RSPG&0.6081&166.3938&2.6216\\
&2-RSPG-V&0.6237&169.8957&2.6340\\\midrule
\multirow{6}{*}{0.01}
&AdaSMSAG&\textbf{0.5872}&\textbf{214.9344}&\textbf{5.8409}\\
&MSNS&27.0953&213.1842&5.0811\\
&M-MDSA&0.6264&187.1180&2.9408\\
&RSPG&30.1006&187.1667&2.6519\\
&2-RSPG&7.6664&196.4444&2.6889\\
&2-RSPG-V&6.6914&196.7941&2.6741\\
\botrule
\end{tabular}
\end{minipage}
\end{center}
\end{table}

\subsection{Wasserstein distributionally robust support vector machine (DRSVM) problem}
We consider a family of Wasserstein distributionally robust support vector machine (DRSVM) problems. The support vector machine (SVM) is one of the most frequently used classification methods and has enjoyed notable empirical successes in machine learning and data analysis \cite{Cauwenberghs,Noble,Durgesh,Rodriguez,Huang,shaoyuanhai,Singla,Cervantes}. However, there are very few literatures addressing the development of fast algorithms for its Wasserstein DRO formulation, which takes the form
\begin{eqnarray}\label{DRSVMorip}
\inf \limits_{w}\left\{\frac{\tau}{2}\|w\|^{2}+\sup\limits_{\mathds{Q} \in B_{\varepsilon}\left(\hat{\mathds{P}}_{n}\right)} \mathds{E}_{(x, y) \sim \mathds{Q}}\left[\ell_{w}(x, y)\right]\right\}.
\end{eqnarray}
Here, $\frac{\tau}{2}\|w\|_{2}^{2}$ is the regularization term with $\tau\geq 0$, $x \in \mathds{R}^{d}$ denotes a feature vector and $y \in\{-1,+1\}$ is the associated label, $\ell_{w}(x, y)=\max \left\{1-y w^{T} x, 0\right\}$ is the hinge loss w.r.t. the random vector $\xi=(x, y)$ and the vector $w \in \mathds{R}^{d}$. The $n$ training samples $\left\{\left(\hat{x}_{i}, \hat{y}_{i}\right)\right\}_{i=1}^{n}$ are independently and identically drawn from an unknown distribution $\mathds{P}^{*}$ on the space $\mathcal{Z}=\mathds{R}^{d} \times\{+1,-1\}$ and $z_{i}=\left(\hat{x}_{i}^T,~ \hat{y}_{i}\right)^T$, $\hat{\mathds{P}}_{n}$ %=\frac{1}{n} \sum_{i=1}^{n} \delta_{\left(\hat{x}_{i}, \hat{y}_{i}\right)}$
is the empirical distribution associated with the training samples. The ambiguity set
\begin{eqnarray*}
B_{\epsilon}\left(\hat{\mathds{P}}_{n}\right)=\left\{\mathds{Q} \in \mathcal{P}(\mathcal{Z}): W\left(\mathds{Q}, \hat{\mathds{P}}_{n}\right) \leq \epsilon\right\}
\end{eqnarray*}
is defined on the space of probability distributions $\mathcal{P}(\mathcal{Z})$ centered at the empirical distribution $\hat{\mathds{P}}_{n}$ and has radius $\epsilon \geq 0$ w.r.t. Wasserstein distance
\begin{eqnarray*}
W\left(\mathds{Q}, \hat{\mathds{P}}_{n}\right)=\inf _{\Pi \in \mathcal{P}(\mathcal{Z} \times \mathcal{Z})}&&\left\{\int_{\mathcal{Z} \times \mathcal{Z}} d\left(\xi, \xi^{\prime}\right) \Pi\left(\mathrm{d} \xi, \mathrm{d} \xi^{\prime}\right): \right.\\
&&~~~\left.\Pi(\mathrm{d} \xi, \mathcal{Z})=\mathds{Q}(\mathrm{d} \xi), \Pi\left(\mathcal{Z}, \mathrm{d} \xi^{\prime}\right)=\hat{\mathds{P}}_{n}\left(\mathrm{d} \xi^{\prime}\right)\right\},
\end{eqnarray*}
where $d\left(\xi, \xi^{\prime}\right)=\Vert x-x^{\prime}\Vert+\frac{\kappa}{2}\vert y-y^{\prime}\vert$ is the transport cost between two data points $\xi, \xi^{\prime} \in \mathcal{Z}$ with $\kappa \geq 0$ representing the relative emphasis between feature mismatch and label uncertainty.

In \cite{AMC}, it is pointed out that \eqref{DRSVMorip} can be reformulated as
\begin{eqnarray}\label{DRSVM}
&\min\limits_{w, \lambda}&\lambda \varepsilon+\frac{\tau}{2}\|w\|^{2}+E_i \left[\max \left\{1-w^{T} z_{i}, 1+w^{T} z_{i}-\lambda \kappa, 0\right\}\right],\nonumber\\
&\textrm{s.t.}&\|w\|\leq \lambda.
\end{eqnarray}
We employ the Neural Networks smoothing function in \cite{Chen1} and obtain the smoothing problem of \eqref{DRSVM} as
\begin{eqnarray}\label{smooDRSVM}
&\min\limits_{w, \lambda}&\lambda \varepsilon+\frac{\tau}{2}\|w\|^{2}+E_i \left[\mu\ln\left(e^{\frac{1-w^{T} z_{i}}{\mu}}+e^{\frac{1+w^{T} z_{i}-\lambda \kappa}{\mu}}+1\right)\right],\nonumber\\
&\textrm{s.t.}&\|w\|\leq \lambda,
\end{eqnarray}
where $\mu$ is the smoothing parameter.

It is well known that the projection on the icecream cone has a closed form formula, which is given as follows:
\begin{lemma}(Theorem 3.3.6 of \cite{Bauschke})
Let $C$ be the icecream cone $\{(x, r) \in X \times \mathds{R}:\|x\| \leq  r\}$. Then for every $(x, r) \in X \times \mathds{R}$:
\begin{eqnarray*}
P_{C}(x, r)=\left\{ \begin{aligned}
&(x,r), &\|x\|\leq r,\\
&(0,0), &r<\|x\|<-r, \\
&\frac{\|x\|+r}{2}\left(\frac{x}{\|x\|}, 1\right), &\|x\|\geq \vert r\vert.
\end{aligned}\right.
\end{eqnarray*}
\end{lemma}

We do numerical experiments on four real datasets described below.

$\bullet$ Pima Indians Diabetes dataset comes from the UCI repository, downloaded from the website\footnotemark[2]\footnotetext[2]{https://archive.ics.uci.edu/ml/datasets/Pima+Indians+Diabetes}. It is a dataset with 768 observations and 8 variables. A population of women who were at least 21 years old, of Pima Indian heritage and living near Phoenix, Arizona, was tested for diabetes according to World Health Organization's criteria. The data were collected by the US National Institute of Diabetes and Digestive and Kidney Diseases.

$\bullet$ Wisconsin breast cancer dataset from the UCI repository (699 patterns) can be downloaded from the website\footnotemark[3]\footnotetext[3]{https://archive.ics.uci.edu/ml/datasets/Breast+Cancer+Wisconsin+(Diagnostic)}. Features are computed from a digitized image of a fine needle aspirate (FNA) of a breast mass. They describe characteristics of the cell nuclei present in the image.

$\bullet$ Image Segmentation(B) is 2-class versions of Image Segmentation dataset, i.e., the first 3 images and the remaining ones. Image Segmentation dataset also comes from the UCI repository\footnotemark[4]\footnotetext[4]{https://archive.ics.uci.edu/ml/datasets/Image+Segmentation}. The instances were drawn randomly from a database of 7 outdoor images. The images were handsegmented to create a classification for every pixel.

$\bullet$ MnistData-10(B) is 2-class versions of MnistData-10 dataset, i.e., the first 5 digits and the remaining ones. Mnist dataset of handwritten digits, available from the website\footnotemark[5]\footnotetext[5]{http://yann.lecun.com/exdb/mnist}. MnistData-10 dataset sampled $10\%$ from Mnist dataset. The digits have been size-normalized and centered in a fixed-size image. The details of the described datasets are resumed in Table \ref{tab:4}.

\begin{table}[h]
\begin{center}
\begin{minipage}{245pt}
\caption{Details of the datasets}\label{tab:4}
\begin{tabular}{@{}cccc@{}}
\toprule
Dataset  & Classes &Sample size &Dimension        \\ \midrule
Pima  & 2       & 768 &8  \\
Wisconsin breast cancer  & 2       & 699 & 10  \\
Image Segmentation(B)  & 2       &2310 & 19  \\
MnistData-10(B)  & 2       &6996 &784 \\
\botrule
\end{tabular}
\end{minipage}
\end{center}
\end{table}

We choose the values of parameters $\tau$ and $\varepsilon$ via the 3-fold cross-validation (CV) using 20 random runs, which are determined by varying them on the grid $\{2,1.5,1,0.5,0.1\}$ and the values with the best average accuracy are chosen for each of the AdaSMSAG, M-MDSA, RSPG, 2-RSPG, and 2-RSPG-V methods. In each problem, $\kappa=1$, $L_f=\frac{\tau}{2}$, and $L=\frac{1}{4}$, and we try to find the $\epsilon$-approximate solution with $\epsilon=0.05$.

We record in Table \ref{tab:7} the parameters $\tau$ and $\varepsilon$ we find by the 3-fold CV, together with the corresponding average accuracy (Acc) of the computed solution and the average CPU time to find the computed solution in seconds, where
\begin{eqnarray*}
\mathrm{Acc}
&=&\frac{\mathrm{the~number~of~correctly~predicted~data\ (\sharp \mathcal{S})}}{\mathrm{the~number~of~total~testing~data}\ (J)}\times 100\%.
\end{eqnarray*}
We can see that our AdaSMSAG method has the best average accuracy in all the datasets. The CPU times of the AdaSMSAG method are the shortest for three datasets. For Wisconsin breast cancer dataset, it is the second shortest one, which is acceptable.

\begin{table}[h]
\begin{center}
\begin{minipage}{270pt}
\caption{The values of parameters $\tau$ and $\varepsilon$, and the Acc and the CPU time determined by 3-fold CV}\label{tab:7}
\begin{tabular}{@{}cccccc@{}}
\toprule
Dataset
& ALG.     & $\tau$ &$\varepsilon$ & Acc & CPU      \\\midrule
\multirow{5}{*}{\begin{tabular}[c]{@{}c@{}}Pima Indians\\Diabetes\end{tabular}}
& AdaSMSAG     &0.1 &0.1& \textbf{0.7560} &\textbf{2.6130}\\
& M-MDSA   &0.1 &0.1& 0.5970 & 2.7961\\
& RSPG     &0.1 &1.5& 0.6074 & 16.3706\\
& 2-RSPG   &0.1 &0.5& 0.6322 & 11.8550\\
& 2-RSPG-V &0.1 &0.1& 0.6251 & 12.3136\\ \midrule
\multirow{5}{*}{\begin{tabular}[c]{@{}c@{}}Wisconsin breast\\cancer\end{tabular}}
& AdaSMSAG     &0.1 &0.1& \textbf{0.9888} &1.1035\\
& M-MDSA   & 0.1&0.1& 0.9470 & 1.1356\\
& RSPG     & 1.5&0.5& 0.9269 & 0.8306\\
& 2-RSPG   & 0.1&0.5& 0.9428 & 12.8805\\
& 2-RSPG-V & 1&2& 0.9381 & \textbf{0.1680}\\  \midrule
\multirow{5}{*}{\begin{tabular}[c]{@{}c@{}}Image\\Segmentation(B)\end{tabular}}
& AdaSMSAG     &0.5 &1.5& \textbf{0.6705} &\textbf{4.0589}\\
& M-MDSA   & 0.1&0.5& 0.4261 & 76.7589\\
& RSPG     & 2&1.5& 0.4859 & 12.3613\\
& 2-RSPG   & 2&2& 0.5579 & 4.1024\\
& 2-RSPG-V & 2&0.1& 0.5620 &5.9542\\  \midrule
\multirow{5}{*}{\begin{tabular}[c]{@{}c@{}}MnistData-10(B)\end{tabular}}
& AdaSMSAG     &2 &2& \textbf{0.7256} &\textbf{26.1302}\\
& M-MDSA   & 0.1&1.5& 0.5829 & 2020.9000\\
& RSPG     & 1.5&1.5& 0.5569 & 240.9487\\
& 2-RSPG   & 1&0.1& 0.6140 & 50.7536\\
& 2-RSPG-V & 1&1.5& 0.6343 & 45.7028\\
\botrule
\end{tabular}
\end{minipage}
\end{center}
\end{table}

\section{Conclusion}
In this paper, we propose an adaptive smoothing mini-batch stochastic accelerated gradient (AdaSMSAG) method for solving nonsmooth convex stochastic composite minimization problems. By using the smoothing approximations, we do not need the requirement that the nonsmooth component is of max structure as in \cite{Nesterov,Wang}. The requirement that the proximal operator of the nonsmooth counterpart does not need also, compare to \cite{Quoc,bian}. The feasible region is no longer restricted to be compact as in \cite{Wang}. The adaptive strategy for decreasing the smoothing parameter is also beneficial to faster computational speed. The convergence result is shown. The order of the worst-case complexity is better than the state-of-the-art methods which is due to the proper use of the structure of the problem by smoothing approximation. Moreover, the efficiency of our AdaSMSAG method has been shown by numerical results on a risk management in portfolio optimization and a family of Wasserstein distributionally robust support vector machine (DRSVM) problems using real datasets.

\hspace{-5mm}\textbf{Funding}~~The work is supported in part by ``the Natural Science Foundation of Beijing, China" (Grant No. 1202021).

\end{document}